\documentclass{article}
\usepackage{amsrefs}
\usepackage{amssymb}
\usepackage{amsmath}
\usepackage{amsthm}
\usepackage{graphicx}

\newtheorem{theorem}{Theorem}
\newtheorem{proposition}{Proposition}
\newtheorem{definition}{Definition}

\DeclareMathOperator{\Cliff}{Cliff}
\DeclareMathOperator{\gon}{gon}
\DeclareMathOperator{\rank}{rk}
\newtheorem*{subject}{2000 Mathematics Subject Classification}
\newtheorem*{keywords}{Keywords}

\theoremstyle{remark}

\author{Marc Coppens\footnote{KU  Leuven, Department of Mathematics, Section of Algebra,
Celestijnenlaan 200B bus 2400 B-3001 Leuven, Belgium; email: marc.coppens@kuleuven.be.}}
\title{The relation between the gonality and the Clifford index of a chain of cycles. }
\date{}

\begin{document}
\maketitle \noindent

\begin{abstract}
For a chain of cycles $\Gamma$ we prove that $\Cliff (\Gamma)=\gon (\Gamma)-2$.
\end{abstract}

\begin{subject}
05C25, 14T15
\end{subject}

\begin{keywords}
Clifford index, gonality, tableaux, metric graphs
\end{keywords}

\section{Introduction}\label{section1}

During the recent decades, a theory of divisors on metric graphs is developed having lots of properties similar to those on smooth projective curves.
For the definitions of this theory and the motivations to develop such a theory we refer to \cite{ref1}.

Using this theory, one can define, on a metric graph $\Gamma$, the gonality $\gon (\Gamma)$ (see Definition {\ref{definition2}) and the Clifford index $\Cliff (\Gamma)$ (see Definition \ref{definition4}) in the same way as in the case of a smooth curve $C$.
As in the case of smooth curves the inequality $\Cliff (\Gamma) \leq \gon (\Gamma)-2$ holds by definition.
In the case of a smooth curve $C$ the inequality $\gon (C)-3 \leq \Cliff (C)$ is proved in \cite{ref2}.
Whether or not this inequality also holds for metric graphs is mentioned as an unknown fact in \cite{ref3}.
The importance of this inequality comes from the fact that $\gon (\Gamma)$ is defined using divisors of rank 1 on $\Gamma$, while $\Cliff (\Gamma)$ gives information on divisors on $\Gamma$ of any prescribed rank.

In this paper we consider this  question for chains of cycles and we prove the following result.

\begin{theorem}\label{theorem1}
If $\Gamma$ is a chain of cycles of genus $g \geq 3$ then $\Cliff (\Gamma)=\gon (\Gamma)-2$
\end{theorem}

Chains of cycles were considered for the first time in \cite{ref4} where the authors used those metric graphs to obtain a tropical proof of the Brill--Noether Theorem.
Later on those metric graphs are used to prove many results in Brill--Noether Theory (see some applications in \cite{ref1}, see e.g. \cite{ref5}).

Although the statement in the theorem is stronger than the result for curves in \cite{ref2}, this stronger statement does not hold for arbitrary metric graphs.
As an example, if $\Gamma$ is a metric graph having the complete graph $K_n$ as an underlying finite graph, then $\gon (\Gamma)=n-1$ while $\Cliff (\Gamma)=n-4$ (see \cite{ref6}, including more results concerning divisors on such graphs).

In the proof of the theorem we rely on the results from \cite{ref7} concerning divisor theory on chains of cycles.
In particular from this theory the theorem is translated into a statement on so-called $\underline{m}$-displacement tableaux (see Theorem \ref{theorem2}).

In the second section we give a short overview on the results of \cite{ref7} and we make a statement (see Theorem \ref{theorem2}) on $\underline{m}$-displacement tableaux implying Theorem \ref{theorem1}.
In the third section we give a proof of that statement.

\section{Some results on divisor theory on chains of cycles}\label{section2}

Let $\Gamma$ be a metric graph of genus $g \geq 2$.

\begin{definition}\label{definition1}
The rank $\rank (D)$ of a divisor $D$ on $\Gamma$ is the largest integer $r$ such that for each effective divisor $E$ of degree $r$ on $\Gamma$ there exists an effective divisor $D'$ on $\Gamma$ equivalent to $D$ containing $E$.
(In case $D$ is not equivalent to an effective divisor then we put $\rank (D)=-1$.)
\end{definition}

\begin{definition}\label{definition2}
The gonality $\gon (\Gamma)$ of $\Gamma$ is the smallest integer $d$ such that there exists a divisor $D$ on $\Gamma$ of degree $d$ such that $\rank (D)\geq 1$.
\end{definition}

On a metric graph $\Gamma$ of genus $g$ we have a canonical divisor $K_{\Gamma}$ and there is a Riemann--Roch Theorem as in the case of curves.
Divisors having rank at most 0 are not interesting, so we are only interested in divisors $D$ satisfying $\rank (D) \geq 1$ and $\rank (K_{\Gamma} -D)\geq 1$.
Because of the Riemann-Roch Theorem this second condition is equivalent to $\rank (D)-d+g\geq 2$.

\begin{definition}\label{definition3}
The Clifford index of a divisor $D$ on $\Gamma$ is $\Cliff (D)=\deg (D)-2 \rank (D)$.
\end{definition}

The Clifford inequality states $\Cliff (D) \geq 0$ if $\rank (D)\geq 1$ and $\rank (D)-d+g \geq 2$.

\begin{definition}\label{definition4}
The Clifford index $\Cliff (\Gamma)$ is the smallest integer $c$ such that there exists a divisor $D$ with $\rank (D) \geq 1$ and $\rank (D)-d+g \geq 2$ with $\Cliff (D)=c$.
\end{definition}

From now on $\Gamma$ is a chain of cycles of genus $g$.
Such a chain is constructed as follows.
Take $g$ graphs $C_1, \cdots, C_g$ each of them isometric to a circle (those are the cycles).
On each cycle $C_i$ we choose two different points $v_i$ and $w_i$.
For $1 \leq i \leq g-1$ the points $w_i$ and $v_{i+1}$ are connected by a straight line segment.
We use the notation $(C_i, v_i, w_i)$ throughout the paper.

\begin{figure}[h]
\begin{center}
\includegraphics[height=2 cm]{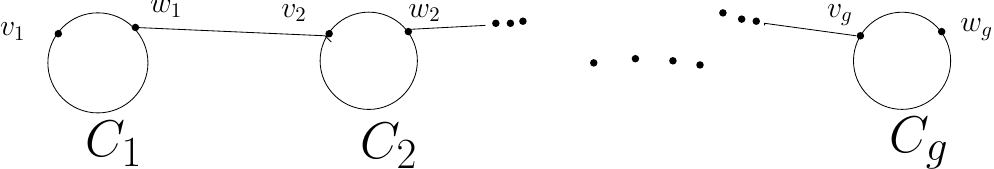}
\caption{a chain of cycles }\label{Figuur 1}
\end{center}
\end{figure}

The length of those segments does not influence the divisor theory on $\Gamma$.
In particular it can be taken equal to zero.
Taking it non-zero is handsome to distinguish e.g. the points $w_1$ on $C_1$ and $v_2$ on $C_2$.
In order to prove Theorem \ref{theorem1} it is enough to prove the following statement on such a graph $\Gamma$.
If $D$ is a divisor of degree $d$ and rank at least $r$ on $\Gamma$ with $r \geq 1$ and $r-d+g \geq 2$ then there exists a divisor $D'$ of degree $d-2r+2$ and rank at least 1 on $\Gamma$.

We now state some definitions and results from \cite{ref7} on chains of cycles $\Gamma$.
Such a chain of cycles $\Gamma$ has torsion profile $\underline{m}=(m_2, \cdots, m_g)$ consisting of $g-1$ non-negative integers defined as follows.

\begin{definition}\label{definition5}
Let $l_i$ be the length of the cycle $C_i$ and let $l(v_i,w_i)$ be the length of the clockwise arc on $C_i$ from $v_i$ to $w_i$.
We put $m_i=0$ if $l_i$ is an irrational multiple of $l(v_i,w_i)$.
Otherwise $m_i$ is the minimal positive integer such that $m_il(v_i,w_i)$ is an integer multiple of $l_i$.
\end{definition}

\begin{definition}\label{definition6}
For positive integers $m$ and $n$ we write $[m \times n]$ to denote the set $\{ 1, \cdots , m \} \times \{ 1, \cdots , n \}$.
Let $\underline{m}=(m_2, \cdots, m_g)$ be a sequence of $g-1$ non-negative integers.
An $\underline{m}$-displacement tableau on $[ m \times n ]$ is a function $t : [m \times n] \rightarrow \{1, \cdots, g \}$ such that
\begin{itemize}
\item $t$ is strictly increasing if one of the two coordinates is fixed.
\item if $t(x,y)=t(x',y')$ then $x-y \equiv x'-y' \mod m_{t(x,y)}$.
\end{itemize}
\end{definition}

From the results in \cite{ref7} we obtain
\begin{proposition}\label{proposition1}
Let $\Gamma$ be a chain of cycles of genus $g$ with torsion profile $\underline{m}$. 
Given $r \in \mathbb{Z}_{\geq 1}$ and $d \in \mathbb{Z}$ with $r-d+g \geq 2$.
There exists a divisor $D$ on $\Gamma$ of degree $d$ and rank at least $r$ if and only if there exists an $\underline{m}$-displacement tableau $t : [(g-d+r) \times (r+1) ] \rightarrow \{ 1, \cdots ,g \}$.
\end{proposition}

So in order to prove Theorem \ref{theorem1} it is enough to prove
\begin{theorem}\label{theorem2}
Let $\underline{m}=(m_2, \cdots, m_g)$ be a sequence of $g-1$ non-negative integers.
Let $r \in \mathbb{Z}_{\geq 1}$ and $d \in \mathbb{Z}$ with $r-d+g \geq 2$.
If there exists an $\underline{m}$-displacement tableau $t : [(g-d+r) \times (r+1) ] \rightarrow \{ 1, \cdots, g \}$ then there exists an $\underline {m}$-displacement tableau $t': [(g-d+2r-1) \times 2] \rightarrow \{ 1, \cdots, g \}$.
\end{theorem}

\section{Proof of the theorem}\label{section3}

In case $r=1$ there is nothing to prove so we are going to assume $r \geq 2$ and we assume the statement of Theorem \ref{theorem2} holds in case $r$ is replaced by any smaller integer.

In the proof we are going to use tableaux obtained from restricting $t$ to subrectangles of $[(g-d+r)\times (r+1)]$. 
They are displacement tableaux using subsequences of $\underline{m}$.
Therefore, for $1 \leq j \leq g$ we introduce $\underline{m}_j=(m_2, \cdots, m_j)$ (for $j=g$ this is $\underline{m}$, for $j=1$ this is empty).
In particular an $\underline{m}_j$-displacement tableau has values in $\{ 1, \cdots, j\}$.

Since $t$ is an $\underline{m}$-displacement tableau we have $g-d+2r \leq g$, equivalently $d \geq 2r$.
In case $d=2r$ we have $t(i,j)=i+j-1$ on $[(g-d+r)\times (r+1)]$.
It follows that $m_2 = \cdots =m_{g-1}=2$.
Then on $[(g-d+2r-1)\times 2] = [ (g-1) \times 2]$ we find $t'(i,j)=i+j-1$ is an $\underline {m}$-displacement tableau.

So, from now on, we can assume $d>2r$.

\

$\underline{\text{Step 1}}$: Theorem \ref{theorem2} holds if $t(g-d+r,r)<g-1$.

If $t(g-d+r,r)<g-1$ then the restriction of $t$ to $[(g-d+r)\times r]=[((g-2)-(d-3)+(r-1))\times ((r-1)+1)]$ is an $\underline {m}_{g-2}$-displacement tableau.
Since $r \geq 2$ of course $r-1 \geq 1$ while $(g-2)-(d-3)+(r-1)=g-d+r \geq 2$, so we can apply the induction hypothesis to that restriction.
We obtain the existence of an $\underline{m}_{g-2}$-tableau $\tilde {t'}$ on $[((g-2)-(d-3)+2(r-1)-1) \times 2]=[(g-d+2r-2)\times 2]$.
We obtain an $\underline{m}$-displacement tableau $t'$ on $[(g-d+2r-1)\times 2]$ having restriction $\tilde {t'}$ to $[(g-d+2r-2)\times 2]$ and taking $t'(g-d+2r-1,1)=g-1$ and $t'(g-d+2r-1,2)=g$ (further on we will say $\tilde {t'}$ is extended).

In case $t(g-d+r,r) \geq g-1$ then $t(g-d+r,r)=g-1$ and $t(g-d+r,r+1)=g$. Because of Step 1, from now on we can assume those two equalities hold.

\

$\underline{\text{Step 2}}$: Theorem \ref{theorem2} holds if $t(g-d+r-1,r+1)=g-1$.

If $t(g-d+r-1,r+1)=g-1=t(g-d+r,r)$ then $m_{g-1}=2$.
The restriction of $t$ to $[(g-d+r)\times r]=[((g-1)-(d-2)+(r-1))\times ((r-1)+1)]$ is an $\underline{m}_{g-1}$-displacement tableau.
Again we can apply the induction hypothesis implying there exists an $\underline{m}_{g-1}$-tableau $\tilde {t'}$ on $[((g-1)-(d-2)+2(r-1)-1) \times 2]=[(g-d+2r-2) \times 2]$.
Because $m_{g-1}=2$ we can extend it to an $\underline{m}$-displacement tableau $t'$ on $[(g-d+2r-1) \times 2]$ taking $t'(g-d+2r-1,1)=g-1$ and $t'(g-d+2r-1,2)=g$.

From now on we can assume also $t(g-d+r-1,r+1)\leq g-2$ (of course $t(g-d+r-1,r+1) >g-1$ is not possible).

This inequality implies $g-d+2r-1 \leq g-2$, equivalently $d \geq 2r+1$.
In case $d=2r+1$ we obtain $t(i,j)=i+j-1$ for $(i,j) \in [(g-d+r-1) \times (r+1)]$.
It follows $m_2 = \cdots = m_{g-3}=2$.
We find an $\underline{m}$-displacement tableau $t'$ on $[(g-d+2r-1) \times 2]=[(g-2) \times 2]$ by taking $t'(i,j)=i+j-1$ in case $(i,j) \in  [(g-3) \times 2]$, $t'(g-2,1)=g-1$ and $t'(g-2,2)=g$.
Therefore from now on we can also assume $d >2r+1$.

\

$\underline{\text{Step 3}}$: Theorem \ref{theorem2} holds if $t(g-d+r-1,r)<g-3$.

If $t(g-d+r-1,r)<g-3$ then the restriction of $t$ to $[(g-d+r-1)\times r]=[((g-4)-(d-4)+(r-1))\times ((r-1)+1)]$ is an $\underline {m}_{g-4}$-displacement tableau.
Since $r \geq 2$ we still have $r-1 \geq 1$.
We have $(g-4)-(d-4)+(r-1)=g-d+r-1$.
In case $g-d+r=2$ we have that $t'(i,j)=t(j,i)$ defines an $\underline{m}$-displacement tableau on $[(r+1)\times 2]=[(2r-r+1) \times 2]=[(2r+g-d-1)\times 2]$, so the theorem is proved in that case.
In case $g-d+r>2$ we have $(g-4)-(d-4)+(r-1)\geq 2$ and we can use the induction hypothesis.
In that case we obtain the existence of an $\underline {m}_{g-4}$-displacement tableau $\tilde {t'}$ on $[((g-4)-(d-4)+2(r-1)-1)\times 2]=[(g-d+2r-3) \times 2]$.
This can be extended to an $\underline{m}$-displacement tableau $t'$ on $[(g-d+2r-1) \times 2]$ by taking $t'(g-d+2r-2,1)=g-3$, $t'(g-d+2r-2,2)=g-2$, $t'(g-d+2r-1,1)=g-1$ and $t'(g-d+2r-1,2)=g$.

So we proved Theorem \ref{theorem2} unless $t(g-d+r-1,r+1)\leq g-2$ and $t(g-d+r-1,r)\geq g-3$.
From the definition of a tableau those inequalities are equivalent to $t(g-d+r-1,r+1)= g-2$ and $t(g-d+r-1,r)= g-3$.
From now on we assume those equalities.

\

For an integer $1 \leq k \leq g-d+r-1$ we formulate the following assumption.

$\underline{\text{Assumption(k)}}$: For all integers $0 \leq k' \leq k$ we have $t(g-d+r-k',r+1)=g-2k'$ and $t(g-d+r-k',r)=g-2k'-1$.

From the previous steps it follows we can assume that Assumption(1) holds.
Also in case Assumption(k) holds then we need $g-2k \leq g-d+2r-k$, or equivalently $d \leq 2r+k$.
In case $d=2r+k$ we have $t(i,j)=i+j-1$ for all $(i,j) \in [(g-d+r-k)\times (r+1)]$
In particular $m_2 = \cdots =m_{g-2k-1}=2$.

\

Now we make an induction step on $k$.

$\underline{\text {Step 4}}$: Assume $k<g-d+r-1$ and $d>2r+k$. If Assumption(k) holds but Assumption(k+1) does not hold then Theorem \ref{theorem2} holds.

The proof of this induction step is similar to steps 2 and 3.

\

Restricting $t$ to $[(g-d+r-k) \times r]=[((g-2k-1)-(d-k-2)+(r-1))\times ((r-1)+1)]$ we obtain an $\underline{m}_{g-2k-1}$-displacement tableau.
We have $r-1\geq 1$ and $(g-2k-1)-(d-k-2)+(r-1)=g-k-d+r \geq 2$.
By assumption we obtain an $\underline{m}_{g-2k-1}$-tableau $\tilde {t'}$ on $[((g-2k-1)-(d-k-2)+2(r-1)-1) \times 2]=[(g-d+2r-k-2) \times 2]$.
In case $t(g-d+r-k-1,r+1)=g-2k-1$ we have $m_{g-2k-1}=2$.
Then we can extend $\tilde {t'}$ by $t'(g-d+2r-k-2+x,1)=g-2k-3+2x$ and $t'(g-d+2r-k-2+x,2)=g-2k-2+2x$ for $1 \leq x \leq k+1$, which gives rise to an $\underline{m}$-displacement tableau on $[(g-d+2r-1)\times 2]$.

To continue the proof we can assume $t(g-d+r-k-1,r+1) \leq g-2k-2$.
In case $t(g-d+r-k-1,r)\leq g-2k-4$ then the restriction of $t$ to $[(g-d+r-k-1) \times r]=[((g-2k-4)-(d-k-4)+(r-1)) \times ((r-1)+1)]$ is an $\underline{m}_{g-2k-4}$-displacement tableau.
We have $r-1 \geq 1$ and $(g-2k-4)-(d-k-4)+(r-1)=g-d-k+r-1 \geq 1$.

In case $g-d-k+r=2$ then we have $t(g-d+r-k,r+1)=t(2,r+1)=g-2k$.
We consider the $\underline{m}_{g-2k}$-displacement tableau $\tilde {t'}$ on $[(r+1)\times 2]=[(2r+1-(d-g+k+2))\times 2]=[(g-d+2r-1-k)\times 2]$ given by $\tilde {t'}(i,j)=t(j,i)$.
This can be extended to an $\underline{m}$-displacement tableau $t'$ on $[(g-d+2r-1)\times 2]$ by taking $t'(g-d+2r-1-i,1)=g-2i-1$ and $t'(g-d+2r-1-i,2)=g-2i$ for $0 \leq i \leq k-1$.

In case $g-d-k+r>2$ then the induction hypothesis gives the existence of an $\underline{m}_{g-2k-4}$-displacement tableau $\tilde {t'}$ on $[((g-2k-4)-(d-k-4)+2(r-1)-1)\times 2]=[(g-d-k+2r-3) \times 2]$.
This can be extended to an $\underline{m}$-displacement tableau $t'$ on $[(g-d+2r-1)\times 2]$ putting $t'(g-d+2r-k-2+x,1)=g-2k-3+2x$ and $t'(g-d+2r-k-2+x,2)=g-2k-2+2x$ for $0 \leq x \leq k+1$.

So we proved Theorem \ref{theorem2} unless $t(g-d+r-k-1,r)\geq g-2k-3$ and $t(g-d+r-k-1,r+1)\leq g-2k-2$.
By definition of a tableau we find equalities in this exceptional case and therefore Assumption(k+1).

\

So we have proved Theorem \ref{theorem2} unless there exists some integer $1 \leq k \leq g-d+r-1$ such that Assumption(k) holds with $d=2r+k$ or $k=g-d+r-1$.

\

$\underline{\text {Step 5}}$: If Assumption(k) holds with $k=d-2r$ then Theorem \ref{theorem2} holds.

It is already noted that in this case $m_2 = \cdots = m_{g-d+2r-k-1}=2$.
So $\tilde {t'}(i,j)=i+j-1$ defines an $\underline{m}_{g-d+2r-k}$-displacement tableau on $[(g-d+2r-1-k)\times 2]$.
This can be extended to become an $\underline{m}$-displacement tableau $t'$ on $[(g-d+2r-1) \times 2]$ taking $t'(g-d+2r-1-i,1)=g-2i-1$ and $t'(g-d+2r-1-i,2)=g-2i$ for $0 \leq i \leq k-1$.

\

$\underline{\text {Step 6}}$: If Assumption(k) holds with $g-d+r-k=1$ then Theorem \ref{theorem2} holds.

Since $t(g-d+r-(g-d+r-2),r+1)=t(2,r+1)=g-2(k-1)$ we obtain an $\underline{m}_{g-2(k-1)}$-displacement tableau $\tilde {t'}$ on $[(r+1)\times 2]$.
But $r+1=2r+1-(d+k-g+1)=g-d+2r-k$ so we can extend $\tilde {t'}$ to an $\underline{m}$-displacement tableau $t'$ on $[(g-d+2r-1) \times 2]$ taking $t'(g-d+2r-1-i,1)=g-2i-1$ and $t'(g-d+2r-1-i,2)=g-2i$ for $0 \leq i \leq k-2$.

This finishes the proof of Theorem \ref{theorem2}.

\

\textbf {Acknowledgement.}  I am grateful to the referee for his suggestions to improve the paper.


\begin{bibsection}
\begin{biblist}


\bib{ref1}{article}{
	author={M. Baker},
	author={D. Jensen},
	title={Degeneration of linear series from the tropical point of view and applications},
	journal={Nonarchimedean and Tropical Geometry, pp. 365-433. Simons Symp., Springer, Cham (2016)},
}
\bib{ref5}{article}{
	author={K. Cook-Powell},
	author={D. Jensen},
	title={Tropical methods in Hurwitz-Brill-Noether theory},
	journal={Advances in Math.},
	volume={398},
	year={2022},
	pages={ Paper No. 108199, 42 pp.},
}
\bib{ref4}{article}{
	author={F. Cools},
	author={J. Draisma},
	author={S. Payne},
	author={E. Robeva},
	title={A tropical proof of the Brill-Noether Theorem},
	journal={Advances in Math.},
	volume={230},
	year={2012},
	pages={759-776},
}
\bib{ref6}{article}{
	author={F. Cools},
	author={M. Panizzut},
	title={The gonality sequence of complete graphs},
	journal={Electr. J. Comb.},
	volume={24(4)},
	year={2017},
	pages={Paper No. 4.1, 20 pp},
}
\bib{ref2}{article}{
	author={M. Coppens},
	author={G. Martens},
	title={Secant spaces and Clifford's Theorem},
	journal={Compositio Math.},
	volume={78},
	year={1991},
	pages={193-212},
}
\bib{ref3}{article}{
	author={A. Deveau},
	author={D. Jensen},
	author={J. Kainic},
	author={D. Mitropolsky},
	title={Gonality of random graphs},
	journal={Involve},
	volume={9},
	year={2016},
	pages={715-720},	
}
\bib{ref7}{article}{
	author={N. Pflueger},
	title={Special divisors on marked chains of cycles},
	journal={J. Combin. Theory Ser. A},
	volume={150},
	year={2017},
	pages={182-207},
}


	

\end{biblist}
\end{bibsection}
\end{document}